\documentclass[12pt]{amsart}
\usepackage[utf8]{inputenc}
\usepackage{amsmath,amssymb,amsthm,color}
\usepackage{graphicx}
\usepackage{a4}

\def\div{\mathop{\rm div}}

\def\eps{\varepsilon}

\def\dd{\,{\rm d}}
\def\div{{\rm div}}

\def\RR{\mathbb{R}}

\def\rhom{\rho_{\min}}
\def\rhoM{\rho_{\max}}

\newtheorem{theorem}{Theorem}

\newtheorem{lemma}[theorem]{Lemma}

\newtheorem{assumption}[theorem]{Assumption}

\newtheorem{remark}[theorem]{Remark}

\begin{document}
\title{Identification of Chemotaxis Models with Volume-Filling}
\author[H. Egger]{Herbert Egger$^\dag$}
\author[J.-F. Pietschmann]{Jan-Frederik Pietschmann$^\dag$}
\author[M. Schlottbom]{Matthias Schlottbom$^\dag$}
\thanks{$^\dag$Numerical Analysis and Scientific Computing, Department of Mathematics, TU Darmstadt, Dolivostr. 15, 64293 Darmstadt. \\}

\begin{abstract}
Chemotaxis refers to the directed movement of cells in response to a chemical signal 
called chemoattractant. A crucial point in the mathematical modeling of chemotactic 
processes is the correct description of the chemotactic sensitivity and of the
production rate of the chemoattractant. In this paper, we investigate the
identification of these non-linear parameter functions in a chemotaxis model with volume-filling.
We also discuss the numerical realization of Tikhonov regularization for the stable solution of the inverse problem. Our theoretical findings are supported by numerical tests.
\end{abstract}

\maketitle

\section{Introduction}

We consider the identification of the parameter functions $f=f(\rho)$ and $g=g(\rho)$ in the 
coupled non-linear parabolic-elliptic system
\begin{align}
 \partial_t \rho &= \div (\nabla \rho - f(\rho)\nabla c)\qquad\text{in } \Omega\times (0,T), \label{eq:rho_our}\\
  -\Delta c + c &=  g(\rho) \quad\text{in } \Omega \times (0,T),\label{eq:c_our}
\end{align}
which is complemented by initial and boundary conditions
\begin{align}
 \rho(0,x) &= \rho_0(x)\quad\text{in}\;\Omega,\label{eq:ic}\\
 \partial_n\rho - f(\rho)\partial_n c  = 0 \quad & \text{and}\quad \partial_n c = 0\,\quad \text{on}\,\partial\Omega\times (0,T).\label{eq:bc} 
\end{align}
The system \eqref{eq:rho_our}--\eqref{eq:c_our} is a non-linear variant of the famous Patlak-Keller-Segel model of chemotaxis 
which describes the motion of bacteria in response to a chemical signal. In this context, $\rho(x,t)$ denotes the bacteria density,
$c(x,t)$ is the concentration of the chemoattractant, $f(\rho(x,t))$ is the chemotactic sensitivity, and $g(\rho(x,t))$ is the production rate of the chemoattractant. 
The boundary conditions in \eqref{eq:bc} describe that there is no flux of bacteria or of the chemoattractant over the boundary $\partial\Omega$, as it is the case in a closed vessel like a petri dish; see \cite{Perthame2007} for further details.

The original model of chemotaxis introduced by Patlak \cite{Patlak1953} and by Keller and Segel \cite{Keller1970,Keller1971} is given by \eqref{eq:rho_our} and a parabolic counterpart of \eqref{eq:c_our} with parameter functions $g(\rho)=\rho$ and $f(\rho)=\chi \rho$ and constant $\chi$. 
For this classical model, solutions can develop blow-up in finite time \cite{Jaeger1992}.
Since blow-up does not appear in biological applications, non-linear variants of the model have been introduced \cite{Burger2006,DiFrancesco2008,Painter2003}. In these models, the chemotactic sensitivity $f(\rho)$ and the production rate $g(\rho)$ are described as non-linear functions of the bacteria density, in particular, $f$ is designed to degenerate at a given maximal density which is referred to as volume-filling, \cite{Wrzosek2010b}. Then global existence of solutions can be established \cite{BurgerDolak-StrussSchmeiser2008}. We will present such a global solvability result below. 
For a review on models and analytical results, let us also refer to \cite{Hillen2009,Horstmann2003,Horstmann2004,Perthame2007}. 

The functions $f$ and $g$ required in the non-linear models of chemotaxis are typically chosen by physical reasoning. The validity of these choices can be tested  by observation of the evolution of the bacteria density $\rho$ in typical petri dish experiments. 
In this paper, we study from an analytical and a numerical point of view the following two important practical questions: \\[-2ex]

(i) Is it possible to uniquely determine $f$ from measurements of $\rho$? \\[-2ex]

(ii) Is it possible to uniquely determine $g$ from measurements of $\rho$? \\[1ex]
We will give affirmative answers to (i) and (ii) in case the other parameter function
is known. 
Note that  $f$ and $g$ only depend on a single variable while measurements 
of $\rho$ will typically be available in space and time. 
The two inverse problems (i) and (ii) are therefore highly overdetermined and one might hope to 
be able to identify both, $f$ and $g$, at the same time. Unfortunately, we cannot give a positive answer to this question yet. For identification results for the parabolic--parabolic case, we refer to Remark \ref{rem:par}.

To the best of our knowledge, only few results on inverse problems in chemotaxis are available to date. 
In \cite{Fister2008} the case $f=f(\rho,c)=\rho\tilde f(c)$ is considered where the function $\tilde f$ is to be identified. The special structure of the cross-diffusion term $\div(\rho\tilde f(c)\nabla c)$ is an important ingredient for the analysis in \cite{Fister2008}, and we need different techniques to prove uniqueness for the inverse problems (i) and (ii) here.
After establishing the identifiability, we also discuss the possibility to reconstruct the parameters by numerical methods. Using the observation of the density $\rho$, we reformulate problem (i) as a linear inverse problem and we investigate Tikhonov regularization for its stable solution; the identification of $g$ could be done in a similar manner. A related approach has been utilized for the identification of hydraulic permeability in groundwater flow in \cite{KaltenbacherSchoeberl02}. The viability of our approach will be demonstrated in numerical experiments.
One could alternatively also formulate Tikhonov regularization for (i) as an
optimization problem constrained by the non-linear pde system
\eqref{eq:rho_our}--\eqref{eq:c_our}. Related optimal control
problems for chemotaxis have been considered in \cite{Feldthordt2012,Fister2003}. 
\\

The outline of the manuscript is as follows: 
In Section~\ref{sec:prelim}, we introduce some basic assumptions and notations that are used throughout the text. 
We prove the existence and uniqueness of solutions to  \eqref{eq:rho_our}--\eqref{eq:bc} in Section~\ref{sec:ex} and also establish regularity and other properties of the solutions that are required for our analysis later on. 
Identifiability of the parameter functions $f$ and $g$ is proven in Section~\ref{sec:inv}. 
The remaining two sections are concerned with the numerical reconstruction of the chemotactic sensitivity $f$. In Section~\ref{sec:reg}, we reformulate the problem as 
a linear inverse problem with perturbed operator, and we discuss its ill-posedness and stable solution by Tikhonov regularization. Section~\ref{sec:num} then presents details of our implementation and numerical tests which support our theoretical results.
We conclude with a few comments on open problems, and, for convenience of the reader, we collect some auxiliary results in a short appendix.

\section{Preliminaries}\label{sec:prelim}

Let $L^p(\Omega)$ denote the Lebesgue spaces of $p$th power integrable
functions with norm $\|u\|_{L^p(\Omega)}=(\int_\Omega |u|^p \dd x)^{1/p}$ for 
$1 \le p < \infty$ and $\|u\|_{L^\infty(\Omega)}=\text{ess}\sup_{x \in \Omega} |u(x)|$. 
The symbol $W^{m,p}(\Omega)$ is used for the Sobolev space of functions in $L^p(\Omega)$ with weak derivatives up 
to order $m$ in $L^p(\Omega)$. The spaces $L^2(\Omega)$ and $W^{1,2}(\Omega)$ are Hilbert spaces and the inner product of $L^2(\Omega)$ is abbreviated by 
\begin{align*}
  (u,v)_\Omega = \int_\Omega u(x) v(x) \dd x
.
\end{align*}
For a Banach space $X$ and $1 \le p \le \infty$, 
we denote by $L^p(0,T;X)$ the Bochner space of functions $u:[0,T]\to X$ with norm
\begin{align*}
  \|u\|_{L^p(0,T;X)}^p = \int_0^T \|u(t)\|_X^p\dd t<\infty.
\end{align*}
For $p=\infty$ the integral is replaced by a essential supremum over $t \in (0,T)$. 
The space $L^2(0,T;L^2(\Omega))$ is again a Hilbert space with inner product
\begin{align*}
  \langle u, v\rangle = \int_0^T (u(t),v(t))_\Omega\dd t%
.
\end{align*}
The following basic assumptions on the domain, the parameters, and the initial condition will be used throughout the text for analyzing the system \eqref{eq:rho_our}--\eqref{eq:bc}.
\begin{assumption}\label{ass:prelim}
(A1) $\Omega \subset \RR^2$ is a bounded domain with $\partial\Omega \in C^{1,1}$.

(A2) $\rho_0\in W^{2-2/p,p}(\Omega)$ for some fixed $2<p<3$, and $0 \le \rho_0 \le 1$ in $\Omega$.

(A3) $f \in W^{1,\infty}(\RR)$ with $f(0) = f(1) = 0$ and $f(\rho) > 0$ for all $\rho\in (0,1)$.

(A4) $g \in W^{1,\infty}(\RR)$ with $g'(\rho) \neq 0$ for a.e.\@ $\rho \in (0,1)$.
\end{assumption}
%
%
Let us shortly discuss these conditions:
We think of a typical petri dish experiment, which motivates our choice of the domain in (A1). 
The box constraints in (A2) can always be satisfied by appropriate scaling. 
The smoothness of $\rho_0$ will be needed below to show regularity of solutions for the system \eqref{eq:rho_our}--\eqref{eq:bc}.
The bound $p>2$ allows us to obtain continuity of $\rho$, and the upper bound $p<3$ is only required to avoid compatibility conditions.
The assumption (A3) ensures that the bacteria density is really sensitive to the concentration of the chemoattractant. The volume-filling condition $f(0)=f(1)=0$ will allow us to establish that any solution of \eqref{eq:rho_our} satisfies $0\leq \rho \leq 1$ for all time. Therefore, the boundedness of $f$ or $g$ is in principle only required on the interval $[0,1]$.
The assumption of monotonicity of the chemotactic production rate $g$ in (A4) ensures that the bacteria always produce (or consume) the chemoattractant. Note that the two equations \eqref{eq:rho_our}--\eqref{eq:c_our} would decouple if $g'\equiv 0$. 

\section{Solvability for the parabolic-elliptic system}\label{sec:ex}

We will now establish existence and regularity of solutions to the parabolic-elliptic system \eqref{eq:rho_our}--\eqref{eq:bc} under weak regularity requirements on the coefficients,
and we will prove uniform a-priori bounds and further properties of the solutions. 
Corresponding results for smooth parameters $f$ and $g$ can be found, e.g., in \cite{Hillen2001}.
%
\begin{theorem}[Existence, uniqueness, regularity]\label{thm:ex_ell} $ $\\
 Let (A1)--(A4) hold. Then for any $T>0$, there exists a unique solution $(\rho,c)$ to \eqref{eq:rho_our}--\eqref{eq:bc} with $\rho \in L^p(0,T;W^{2,p}(\Omega))\cap W^{1,p}(0,T;L^p(\Omega))$ and $c\in L^\infty(0,T;W^{2,p}(\Omega))$. Moreover, there holds
  \begin{align*}
     \|\rho\|_{L^p(0,T;W^{2,p}(\Omega))}+\|\partial_t\rho\|_{L^p(0,T;L^p(\Omega))} + \|c\|_{L^\infty(0,T;W^{2,p}(\Omega))} \leq C \|\rho_0\|_{W^{1,p}(\Omega)},
  \end{align*}
with $C$ depending only on the domain and the bounds for the coefficients. 
Since $p>2$, we also have $\rho,c \in C([0,T] \times \overline\Omega)$ and $\nabla c \in C([0,T]\times \overline\Omega)^2$ by embedding.
\end{theorem}
\begin{proof}
We first establish local existence of solutions via Banach's fixed point theorem. 
Consider the non-empty and closed set
  \begin{align*}
     \mathcal{M}=\{\rho \in L^\infty(0,T;L^2(\Omega)):\ \|\rho\|_{L^\infty(0,T;L^2(\Omega))}\leq C_\mathcal{M}\}.
  \end{align*}
The constants $C_\mathcal{M}$ and $T>0$ will be specified below.
On $\mathcal{M}$ we define the mapping
  \begin{align*}
    \Phi: \mathcal{M} \to L^\infty(0,T;L^2(\Omega)), \qquad \tilde \rho \mapsto \rho,
  \end{align*}
  where $\rho$ is the weak solution of the linearized system
  \begin{align}\label{eq:rho_lin}
    \partial_t \rho - \Delta \rho &= -\div(f(\tilde \rho)\nabla c) \quad\text{in }\Omega \times (0,T),\\
    - \Delta c + c &= g(\tilde \rho) \quad\text{in }\Omega \times (0,T) \label{eq:c_lin},
 \end{align}
which is complemented by homogeneous Neumann conditions $\partial_n \rho =\partial_n c=0 $ on $\partial\Omega\times(0,T)$ and the initial condition $\rho(x,0)=\rho_0(x)$ for $x\in\Omega$. 
Using assumption (A4) and Lemma~\ref{lem:ex_lin_ell} with $h=g(\tilde \rho)\in L^\infty(0,T;L^\infty(\Omega))$ 
the solutions of \eqref{eq:c_lin} can be shown to be uniformly bounded in $L^\infty(0,T;W^{2,p}(\Omega))$ for all $T>0$. In particular, since $p>2$, we obtain $\|\nabla c\|_{L^\infty(0,T;L^\infty(\Omega))} \le C_c$ for some constant $C_c$ independent of $T$.
By assumption (A3) we further obtain that $h=f(\tilde \rho) \nabla c$ is uniformly bounded in $L^\infty(0,T;L^2(\Omega))$,
and an application of Lemma~\ref{lem:ex_lin_par_l2} yields that the solution of \eqref{eq:rho_lin} is bounded by
\begin{align*}
 \| \rho\|_{L^\infty(0,T;L^2(\Omega))} \leq \| f(\tilde \rho)\nabla c \|_{L^2(0,T;L^2(\Omega))} + \|\rho_0\|_{L^2(\Omega)}\\
  \leq \sqrt{T|\Omega|} \|f\|_{L^\infty(\RR)} \|g\|_{L^\infty(\RR)}+ \|\rho_0\|_{L^{2}(\Omega)}=:C_\mathcal{M}.
\end{align*}
This shows that $\Phi$ is a self-mapping on $\mathcal{M}$ if $C_\mathcal{M}$ is chosen appropriately. 

Now let $\rho_1 = \Phi(\tilde \rho_1)$ and $\rho_2=\Phi(\tilde \rho_2)$ with $\tilde \rho_1$, $\tilde \rho_2\in \mathcal{M}$. 
Then  $\rho_1-\rho_2$ satisfies the coupled linear parabolic-elliptic system
\begin{align*}
 \partial_t (\rho_1 - \rho_2) - \Delta (\rho_1-\rho_2) &= -\div(f(\tilde \rho_1) \nabla c_1 - f(\tilde \rho_2) \nabla c_2) \\
-\Delta (c_1-c_2) + (c_1-c_2) &= g(\tilde \rho_1) - g(\tilde \rho_2)
\end{align*}
with homogeneous initial and boundary conditions. 
From the second equation and Lemma~\ref{lem:ex_lin_ell} with  $h=g(\tilde \rho_1) - g(\tilde \rho_2) = g'(\hat \rho) (\tilde \rho_1-\tilde \rho_2)$, we deduce
\begin{align*}
\|c_1-c_2\|_{L^\infty(0,T;W^{1,2}(\Omega))} \le \|g'\|_{L^\infty(\RR)} \|\tilde \rho_1-\tilde \rho_2\|_{L^\infty(0,T;L^2(\Omega))}. 
\end{align*}
Applying Lemma~\ref{lem:ex_lin_par_l2} with $h=f(\tilde\rho_1)\nabla c_1 - f(\tilde\rho_2)\nabla c_2$,
we obtain the following estimate for the solution of the first equation 
\begin{align*}
 &\|\rho_1-\rho_2\|_{L^\infty(0,T;L^2(\Omega))} \le \|f(\tilde\rho_1)\nabla c_1-f(\tilde\rho_2)\nabla c_2\|_{L^2(0,T;L^2(\Omega))} \\
&\qquad \qquad 
  \le \|(f(\tilde\rho_1)-f(\tilde\rho_2))\nabla c_1\|_{L^2(0,T;L^2(\Omega))} 
    + \|f(\tilde\rho_2)(\nabla c_1-\nabla c_2)\|_{L^2(0,T;L^2(\Omega))}.
\end{align*}
Due to the uniform a-priori bound for $c_1$, the first term can be estimated by
\begin{align*}
\|(f(\tilde\rho_1)-f(\tilde\rho_2))\nabla c_1\|_{L^2(0,T;L^2(\Omega))} 
    &\leq  \sqrt{T} C_c \|f'\|_{L^\infty(\RR)} \|\tilde \rho_1-\tilde\rho_2\|_{L^\infty(0,T;L^2(\Omega))},
\end{align*}
and using the previous estimate for $c_1-c_2$, we obtain the bound
\begin{align*}
   \|f(\tilde\rho_2)(\nabla c_1-\nabla c_2)\|_{L^2(0,T;L^2(\Omega))} 
   &\leq \sqrt{T} \|f\|_{L^\infty(\RR)}   \|g'\|_{L^\infty(\RR)}\|\tilde \rho_1-\tilde\rho_2\|_{L^\infty(0,T;L^2(\Omega))}
 \end{align*}
for the second term. Combining these estimates with the one for $\rho_1-\rho_2$, we get
  \begin{align*}
    \|\rho_1-\rho_2\|_{L^\infty(0,T;L^2(\Omega))} \leq  C' \sqrt{T} \|\tilde \rho_1-\tilde\rho_2\|_{L^\infty(0,T;L^2(\Omega))},
  \end{align*}
  where $C'$ only depends on $\Omega$ and the bounds for the coefficients.
  Choosing $T$ small enough, we conclude that $\Phi$ is a contraction on $\mathcal{M}$.

Hence by Banach's fixed point theorem, there exists a unique $\rho \in \mathcal{M}$ such that $\rho = \Phi(\rho)$. Applying Lemma~\ref{lem:ex_lin_ell} with $h=g(\rho)$ and Lemma~\ref{lem:ex_lin_par_l2} with $h=\div(f(\rho) \nabla c)$, we see that $\rho\in L^2(0,T;W^{1,2}(\Omega))$. 
We can now differentiate the right hand side of \eqref{eq:rho_lin} and rearrange terms to realize that $\rho$ also satisfies
\begin{align*}
\partial_t \rho - \Delta \rho + f'(\rho) \nabla c \cdot \nabla \rho  &= -f(\rho) \Delta c.
\end{align*}
This amounts to problem \eqref{eq:lin_par_lp}--\eqref{eq:lin_par_bc_lp} with
    $b = f'(\rho)\nabla c\in L^\infty(0,T;L^\infty(\Omega))$ and $h=-f(\rho)\Delta c \in L^\infty(0,T;L^p(\Omega))$.
By Lemma~\ref{lem:ex_lin_par_lp}, we can thus conclude that $\rho\in L^p(0,T;W^{2,p}(\Omega))\cap W^{1,p}(0,T;L^p(\Omega))$.
%
Due to the uniform boundedness of $g$ and thus of $c$, the existence and regularity result can be made global in time by a standard continuation argument.
\end{proof}

In addition to the a-priori estimates of the previous theorem, we will also require pointwise bounds on the solution $\rho$ for our analysis of the inverse problems.  
The following result strongly relies on the volume-filling property of our model, 
i.e. the condition $f(0)=f(1)=0$ in assumption (A3). 
For smooth functions $f$, a similar statement, but with a different proof, can be found in \cite{Hillen2001}.
\begin{lemma}[Invariant Regions]\label{lem:invariant} $ $\\
  Let (A1)--(A4) hold and let $(\rho,c)$  be a regular solution of the system \eqref{eq:rho_our}--\eqref{eq:bc} with $\rho\in L^p(0,T;W^{2,p}(\Omega))\cap W^{1,p}(0,T;L^p(\Omega))$ and $c\in  L^\infty(0,T;W^{2,p}(\Omega))$. Then
  $$
  0\leq \rho(x,t) \leq 1\qquad \text{for all } (x,t)\in \Omega \times (0,T).
  $$
\end{lemma}
\begin{proof}
 For $\gamma>0$ let us define $\eta_\gamma\in W^{2,\infty}(\RR)$ by
\begin{align*}
 \eta_\gamma(\rho) = \left\{\begin{array}{ll}
                      0, & \rho \le 0,\\
		      \frac{\rho^2}{4\gamma},& 0 < \rho \le 2\gamma,\\
		      \rho-\gamma,&\rho > 2\gamma,
                     \end{array}\right.
                     \quad\text{with }\quad
                      \eta_\gamma''(\rho)= \begin{cases}
                        0,                      & \rho \le 0,\\
                        \frac{1}{2\gamma},      & 0 < \rho \le 2\gamma,\\
                        0,                      &\rho > 2\gamma.
                      \end{cases}
\end{align*}
Note that $\eta_\gamma(\rho)$ is a regularization of the function $\rho_+=\max(\rho,0)$. 
Using equations \eqref{eq:rho_our} and \eqref{eq:bc}, integration-by-parts, and Young's inequality, we obtain
\begin{align}\nonumber
 \frac{d}{dt}\int_\Omega \eta_\gamma(\rho - 1)\dd x &= \int_\Omega \eta_\gamma'(\rho-1)\partial_t \rho\dd x= \int_\Omega \eta_\gamma'(\rho-1) \; \div(\nabla\rho - f(\rho)\nabla c)\dd x\\
&=-\int_\Omega \eta_\gamma''(\rho-1)\left(|\nabla\rho|^2 - f(\rho)\nabla c\cdot\nabla\rho\right)\dd x \nonumber\\
&\le -\frac{1}{2}\int_\Omega \eta_\gamma''(\rho-1)|\nabla\rho|^2\dd x + \frac{1}{2}\int_\Omega \eta_\gamma''(\rho-1) f(\rho)^2 |\nabla c|^2\dd x.\label{eq:pos_est}
\end{align}
We claim now that the last integral vanishes when we let $\gamma\to 0$. To see this, we define $\Omega_\gamma = \{x\in\Omega:\ 1\le\rho(x,t)\le 1+2\gamma\}$ and use $f(1)=0$, to get
\begin{align*}
&\int_{\Omega_\gamma} \eta_\gamma''(\rho-1)(f(\rho)-f(1))^2|\nabla c|^2\dd x \\
&\le \|f'\|_{L^{\infty}(\RR)}^2\int_{\Omega_\gamma} \frac{(\rho-1)^2}{2\gamma} |\nabla c|^2\dd x 
 \le  2\gamma \|f'\|_{L^{\infty}(\RR)}^2 \|\nabla c\|_{L^\infty(0,T;L^2(\Omega))}^2.
\end{align*}
Together with the non-positivity of the first term in \eqref{eq:pos_est}, we conclude that
\begin{align*}
 \frac{d}{dt}\int_\Omega (\rho - 1)_+\dd x = \lim_{\gamma\to 0^+}\frac{d}{dt}\int_\Omega \eta_\gamma(\rho - 1)\dd x \leq 0.
\end{align*}
Using the box constraints in (A2) for the initial density, we thus obtain that
\begin{align*}
 0 \leq \int_\Omega (\rho - 1)_+\dd x \leq \int_\Omega (\rho_0 - 1)_+\dd x\leq 0,
\end{align*}
for every $t \ge 0$. This implies that $(\rho-1)_+ = 0$, i.e. $\rho \le 1$ on $\Omega \times (0,T)$. The other direction $0\le \rho$ follows with the same arguments, by considering $(\rho)_-=(-\rho)_+$ instead of $(\rho-1)_+$ and using $f(0)=0$ instead of $f(1)=0$.
\end{proof}
\section{Uniqueness for the inverse problems}\label{sec:inv}

We are now in a position to address the two identification problems outlined in the introduction: 
Can the observation of the bacteria density $\rho$ on $\Omega \times (0,T)$ be used to uniquely determine either
\begin{enumerate}
 \item[(i)] the chemotactic sensitivity $f$, or
 \item[(ii)] the production rate $g$ of the chemoattractant,
\end{enumerate}
if the other of the two parameter functions is known?
Note that identification is of course only possible on the interval $(\rhom,\rhoM)$ of densities that are attained; 
here $\rhom=\min_{(x,t) \in \overline{\Omega} \times [0,T]} \rho(x,t)$ and $\rhoM=\max_{(x,t) \in \overline{\Omega} \times [0,T]} \rho(x,t)$. 


\subsection{Identification of $f$}\label{sec:id_f}
Denote by $(\rho_1,c_1)$ and $(\rho_2,c_2)$ the solutions of \eqref{eq:rho_our}--\eqref{eq:bc}
with $f$ replaced by $f_1$ and $f_2$, respectively. We then have 
\begin{theorem}\label{thm:uniqueness}
  Let (A1), (A2), (A4) hold, and let $f_1,f_2$ satisfy (A3). Then
  \begin{align*}
    \rho_1=\rho_2 \text{ on } \Omega \times (0,T) \quad \text{ implies } \quad f_1 = f_2 \text{ on } (\rhom,\rhoM).
  \end{align*}
\end{theorem}
\begin{proof}
If $\rho_0$ is constant, then $\rho$ and $c$ are constant for all time and $(\rhom,\rhoM)$ is empty, so nothing has to be shown. 
We therefore assume from now on that $\rho_0$ is not constant and we rewrite equation \eqref{eq:rho_our} as 
  \begin{align}\label{eq:i12}
  -\div(f_i(\rho_i)\nabla c_i) = \partial_t \rho_i - \Delta \rho_i,\quad i=1,2.
  \end{align}
  Since $\rho_1=\rho_2 =: \rho$, equation \eqref{eq:c_our} implies that $c_1=c_2 =: c$. 
  We then subtract the two equations \eqref{eq:i12} for $i=1,2$, to obtain that
  \begin{align}\label{eq:Tf_0}
    -\div((f_1(\rho)-f_2(\rho))\nabla c) = 0\qquad\text{on } \Omega \times (0,T).
  \end{align}
  This is a linear equation in $F=f_1-f_2$ and it remains to show that \eqref{eq:Tf_0} implies $F(\rho)=0$ for all $\rho\in (\rhom,\rhoM)$. 
  We argue by contradiction: 

  Assume that there exists $\bar\rho\in (\rhom,\rhoM)$ with $F(\bar\rho)>0$ and $\bar\rho=\rho(\bar x,\bar t)$ for some $(\bar x,\bar t)\in\Omega \times (0,T)$. 
  Since $F(\rho)_+ =\max\{F(\rho),0\}= F(\rho)$ on the open and nonempty set $U = \{(x,t)\in\Omega \times (0,T):\ F(\rho(x,t))> 0\}$, we infer from \eqref{eq:Tf_0} that
   \begin{align*}
    -\div(F(\rho)_+\nabla c) = 0\qquad\text{in } U.
  \end{align*}
  Without loss of generality, we may assume that $\overline U \; \cap \; (\Omega \times \{0\})$ is not empty; otherwise 
  we can exchange the role of $f_1$ and $f_2$.  Multiplying this equation by the concentration $c$ and integrating over $U$ yields
  \begin{align*}
   0 &= -\iint_{U} \div(F(\rho)_+\nabla c) c\d(x,t) = -\int_0^T\int_{\Omega} \div(F(\rho)_+\nabla c) c\dd x\dd t\\
   &= \iint_{U} F(\rho)_+ |\nabla c|^2\dd (x,t).
  \end{align*}
  In the last step we used integration-by-parts and, respectively, the boundary condition $\partial_n c=0$ on $\partial\Omega \times (0,T)$ to eliminate the boundary term.  Since $F(\rho)_+ = F(\rho)>0$ on $U$, we infer that
  \begin{align}\label{eq:null}
    \nabla c = 0\qquad \text{on } U,
  \end{align}
  from which we also conclude that $\Delta c=0$ on $U$. Using this in equation \eqref{eq:c_our} 
  we obtain by differentiation
  \begin{align*}
    0 = \nabla c = g'(\rho)\nabla\rho \qquad \text{on } U,
  \end{align*}
  and from assumption (A4) we deduce that $\nabla \rho = 0$ on $U$. 
  Inserting this in equation \eqref{eq:rho_our}, we also obtain that $\partial_t\rho=0$ on $U$. 
  Thus, $\rho$ is constant on every connected component of $U$, and by continuity also on $\overline U$, which is a contradiction to $\rho_0 \neq const$. Therefore, $F(\rho) = f_1(\rho) - f_2(\rho) = 0$ on $\Omega \times (0,T)$.
%
\end{proof}
\subsection{Identification of $g$}\label{sec:id_g}
Let us now turn to the problem of identifying the chemotactic production rate $g$ when the chemotactic sensitivity $f$ is known. Here we denote by $(\rho_1,c_1)$ and $(\rho_2,c_2)$ the solutions of the system \eqref{eq:rho_our}--\eqref{eq:bc} with $g$ replaced by $g_1$ and $g_2$, respectively. For this case, we have

\begin{theorem}\label{thm:uniquenes_g_time}
  Let (A1)--(A3) hold, and assume that $g_1,g_2$ satisfy (A4). Then 
  \begin{align*}
    \rho_1=\rho_2 \text{ in } \Omega \times (0,T) \quad \text{ implies } \quad g_1 = g_2 + C \text{ on } (\rhom,\rhoM)
  \end{align*}
for some constant $C \in \RR$ that cannot be identified.
\end{theorem}
\begin{proof}
  We set $\rho := \rho_1=\rho_2$ and subtract equation \eqref{eq:rho_our} for $c_1$ and $c_2$ to obtain
  \begin{align*}
   -\div(f(\rho)\nabla(c_1-c_2)) = 0\qquad \text{on } \Omega \times (0,T).
  \end{align*}
  Multiplying this equation by $c_1-c_2$, integrating over the domain $\Omega$, integrating by parts, and using the boundary conditions \eqref{eq:bc}
  yields
  \begin{align*}
    \int_\Omega f(\rho) |\nabla(c_1-c_2)|^2\dd x = 0 \quad\text{for all } t\in (0,T).
  \end{align*}
  This further implies that
  \begin{align}\label{eq:nabla_c_zero}
    \nabla(c_1-c_2) = 0\quad\text{on } U= \{ (x,t)\in\Omega \times (0,T):\ f(\rho(x,t))>0\}.
  \end{align}
  By continuity of $\rho$ and by $f(\tilde \rho)>0$ for all $0<\tilde\rho<1$ due to (A3), there exists an open connected component $V$ of $U$ with 
  $(\rhom,\rhoM) = \{\rho(x,t) : (x,t) \in V\}$. Because of \eqref{eq:c_our} and \eqref{eq:nabla_c_zero} we get 
  \begin{align} \label{eq:dt}
    g_1(\rho(x,t))-g_2(\rho(x,t)) = c_1(x,t) - c_2(x,t) = d(t) \quad\text{for all }(x,t)\in V
  \end{align}
  with some continuous function $d$ depending only on $t$. We will show below, that $d$ is in fact constant on $V$, which by \eqref{eq:dt} and the fact that $\rho$ attains all possible values on $V$ yields the assertion of the theorem.

  Let us now show that $d$ is constant on $V$: We denote by $[t_0,t_1]$ the smallest interval such that $V \subset \Omega \times (t_0,t_1)$ and set $V_t = \{x \in \Omega : (x,t) \in V\}$. 
  First assume that $\rho(\cdot,\bar t) \equiv const$ for some $\bar t \in (t_0,t_1)$: Then $V_{\bar t} = \Omega$ and $\rho(\cdot,t) = \rho(\cdot,\bar t) \equiv const$ for all $t \ge \bar t$, and also $d(t) = d(\bar t)$ for all $t \ge \bar t$.   
  Now assume that $\rho(\cdot,\bar t) \not\equiv const$ on $V_{\bar t}$. Then there exists $\bar x \in V_{\bar t}$ and $\eps>0$ such that $\bar \rho=\rho(\bar x, \bar t)$ and $(\bar \rho-\eps,\bar \rho + \eps) \subset \{\rho(x,\bar t) : x \in V_{\bar t}\}$. Since $V$ is open and $\rho$ is continuous, there exists $\delta>0$ such that the ball $B_\delta(\bar x,\bar t) \subset V$ and $\rho(x,t) \in (\bar \rho-\eps,\bar\rho+\eps)$ for all $(x,t) \in B_\delta(\bar x,\bar t)$. From this and \eqref{eq:dt} we conclude that $d(t)=d(\bar t)$ for all $|t-\bar t|<\delta$. Using a continuation argument, we obtain that $d(t)=d(\bar t)$ for all $t \in (t_0,t_1)$, which was to be shown.
  
  It can easily be seen, that a shift of $g(\rho)$ by a constant value just shifts $c$ by a constant value and therefore does not change the density $\rho$. Therefore, $g$ can at most be identified up to constants.
\end{proof}
 
\begin{remark}[parabolic--parabolic case]\label{rem:par} Let us also briefly comment of identifiability for the parabolic--parabolic system given by \eqref{eq:rho_our} and 
$$
\partial_t c - \Delta c + c = g(\rho),
$$ 
instead of \eqref{eq:c_our}. We expect that the proofs for both the unique identifiability of $f$ and $g$ can be adapted to this case. In fact, for Theorem \ref{thm:uniqueness} the only modification is to notice that $\nabla c = 0$ also implies $\partial_t \nabla c = 0$. For Theorem \ref{thm:uniquenes_g_time}, the proof will remain unchanged until the definition of $d(t)$ which will contain an additional additive term stemming from the time derivatives. To give precise statements together with an adapted existence theory and the numerical treatment of the reconstruction of $g$ is work in progress.
\end{remark}

\section{Forward Operator -- Ill-posedness -- Regularization}\label{sec:reg}

In this section we study  in more detail the inverse problem of determining the unknown chemotactic sensitivity $f$ from observation of the bacteria density $\rho(x,t)$ for all $(x,t)\in \Omega \times (0,T)$.
Let us denote by $f^0$ the true chemotactic sensitivity and by $\rho^0$, $c^0$ the corresponding solution of the system \eqref{eq:rho_our}--\eqref{eq:bc}.
In view of the results of Section~\ref{sec:id_f} the data $\rho^0$ contain enough information to identify $f^0$ uniquely
on the interval $[\rho_{min},\rho_{max}]$ of values of the density that is attained in the experiment.
In practice we have to deal with noisy data $\rho^\delta$, for which we assume that 
\begin{align}
  \| \rho^0-\rho^\delta\|_{L^2(0,T;L^2(\Omega))} \leq \delta. \label{eq:noise}
\end{align}
As usual, the noise level $\delta$ is assumed to be known.
%
Using the observation $\rho^\delta$, we can define a perturbed forward operator
\begin{align*}
  T^\delta: H^1(0,1)\to L^2(0,T;L^2(\Omega)),\ f \mapsto r^\delta
\end{align*}
where $r^\delta \in L^2(0,T;W^{1,2}(\Omega))\cap W^{1,2}(0,T;W^{1,2}(\Omega)')$ is a solution to the system
  \begin{align}
    \partial_t r^\delta - \Delta r^\delta &= -\div( f( \rho^\delta)\nabla c^\delta) \quad\text{in }\Omega \times (0,T),\label{eq:rho_IP}\\
    - \Delta c^\delta + c^\delta &= g( \rho^\delta) \quad\text{in }\Omega \times (0,T),\label{eq:c_IP}
 \end{align}
complemented by homogeneous Neumann conditions on $\partial \Omega \times (0,T)$ and the initial condition $r^\delta(0)=\rho_0$ in $\Omega$. In view of Lemmas~\ref{lem:ex_lin_par_l2} and \ref{lem:ex_lin_ell}, 
the mapping $T^\delta$ is well-defined.
The inverse problem of identifying $f$ can then be formulated as
\begin{align}\label{eq:IP}
  T^\delta f = \rho^\delta.
\end{align}
We denote by $T$ the operator with $\rho^0$ used instead of $\rho^\delta$ in the right hand side of equations~\eqref{eq:rho_IP} and \eqref{eq:c_IP}. Then $Tf^0 =\rho^0$, so a solution for unperturbed data and operator exists.
Next, let us summarize some basic properties of the forward operator.
\begin{lemma}
  For any $\delta \ge 0$, the operator $T^\delta:H^1(0,1)\to L^2(0,T;L^2(\Omega))$ is affine linear, bounded, and compact.
\end{lemma}
\begin{proof}
 Affine linearity is clear, and compactness of $T^\delta$, and hence boundedness, is a direct consequence of the Aubin-Lions lemma \cite{AubinLions}.
\end{proof}
%
%
As a direct consequence of the compactness of $T$ and $T^\delta$, the inverse problem is ill-posed, and some sort of regularization is required.

\subsection{Regularization}

In the following, we consider Tikhonov regularization for a stable solution of the perturbed inverse problem \eqref{eq:IP}. For $\alpha>0$, we define regularized approximations via the minimization problem
\begin{align}\label{eq:Tikhonov}
J_\alpha^\delta(f) = 
  \frac{1}{2} \| T^\delta f - \rho^\delta\|_{L^2(0,T;L^2(\Omega))}^2 + \frac{\alpha}{2} \|f\|_{H^1(0,1)}^2 \to \min_{f\in H^1(0,1)}!
\end{align}
From standard regularization theory \cite[Chapter 5]{EHN96}, we know that \eqref{eq:Tikhonov} has a unique minimizer $f_\alpha^\delta$ for any $\alpha>0$. 
To show convergence of $f_\alpha^\delta$ towards the solution, we also need an estimate for the perturbation in the operator.
%
With the same arguments as in the proof of Theorem~\ref{thm:ex_ell} one can see that 
\begin{align} \label{eq:perturbation}
  \|T^\delta f - &T f\|_{L^2(0,T;L^2(\Omega))} \leq C \delta \|f\|_{W^{1,\infty}(0,1)} \quad\text{for any $f \in W^{1,\infty}(0,1)$.}
 \end{align}
It is also possible to bound the perturbation error by $C \delta^{1/2} \|f\|_{H^{1}(0,1)}$ for all functions $f \in H^1(0,1)$.
Using the results of \cite[Chapter 5]{EHN96}, one can show that 
$f_\alpha^\delta$ converges to the minimum-norm solution $f^\dag$ of the unperturbed problem $T f^0 = \rho^0$, i.e., to the solution 
of minimal $H^1$-norm. Note that for the unperturbed problem such a solution always exists. 
For convenience of the reader, let us state the basic convergence result explicitly.
\begin{lemma} \label{lem:tik}
Let $T : X \to Y$ be a bounded linear operator between Hilbert spaces $X$ and $Y$, and $\rho \in R(T)$.
For $\delta>0$ let $\|\rho-\rho^\delta\| \le \delta$ and let $T^\delta : X \to Y$ be bounded linear operators with $\|Tf^\dag - T^\delta f^\dag\|_{Y} \le C(f^\dag) \delta$.
Then the regularized solutions $f_\alpha^\delta$ 
defined by \eqref{eq:Tikhonov} converge to the minimum-norm solution $f^\dag$ of $T f = \rho$ with $\delta \to 0$,
provided that $\alpha \to 0$ and $\delta^2/\alpha \to 0$.
\end{lemma}
\def\tT{\widetilde T}
\begin{proof}
To avoid double superscripts, let us write $\tT$ for $T^\delta$.
Using $T f^\dag=\rho$ one easily obtains 
\begin{align*}
(\tT^* \tT + \alpha I) (f_\alpha^\delta - f^\dag) 
&= \tT^* (\rho^\delta - \rho) + \tT^* (\tT f^\dag - T f^\dag) - \alpha f^\dag. 
\end{align*}
Applying the inverse of $\tT^* \tT + \alpha I$ and the triangle inequality yields 
\begin{align*}
\|f_\alpha^\delta - f^\dag\| 
\le \|(\tT^* \tT + \alpha I)^{-1} \tT^*\| (\|\rho^\delta - \rho\| + \|\tT f^\dag - T f^\dag\|) + \|\alpha (\tT^* \tT + \alpha I)^{-1} f^\dag\|. 
\end{align*}
By the usual spectral estimates, we get $\|(\tT^* \tT + \alpha I)^{-1} \tT^*\| \le \alpha^{-1/2}$ and $\|\alpha (\tT^* \tT + \alpha I)^{-1} f^\dag\| \to 0$ with $\alpha \to 0$. The assertion then directly follows from the assumptions and the conditions on $\alpha$ and $\delta$.
\end{proof}
From our uniqueness results we can deduce that $f^\dag = f^0$ on $[\rho_{min},\rho_{max}]$, where $f^0$ is the true solution.
On the remaining part of the interval $[0,1]$, the minimum norm solution solves $-\Delta f^\dag + f^\dag=0$. 
Hence, $f^\dag$ is in $W^{1,\infty}$ globally. In view of \eqref{eq:perturbation} Lemma~\ref{lem:tik} thus applies almost verbatim to our problem.
As can be seen from the proof, one can also obtain quantitative estimates in the usual manner. 
In our numerical examples, we utilize the discrepancy principle as a parameter choice rule, 
i.e. we choose the maximal $\alpha>0$ such that
\begin{align}\label{eq:discrepancy}
  \| T^\delta f_{\alpha}^\delta - \rho^\delta\|_{L^2(0,T;L^2(\Omega))}\leq \tau \delta
\end{align}
for some appropriate $\tau>1$. 
Assuming that the minimum-norm solution satisfies an appropriate source condition, we can expect that $\|f_\alpha^\delta-f^\dagger\|_{H^1(0,1)}=\mathcal{O}(\sqrt{\delta})$, which is what we observe in our numerical tests.

\section{Numerical Examples}

\subsection*{Setup}
To mimic a typical experiment in a petri dish, we choose $\Omega = B_1(0)\subset\RR^2$. 
For our numerical test, we set
\begin{align*}
 f^0(\rho) = \rho(1-\rho)\;\text{and}\;g(\rho)=\rho,
\end{align*}
which is a typical form of the parameters that can be found in the literature.
Furthermore, we define the initial datum by
\begin{align}\label{eq:rho0}
 \rho_0(x) = 0.45\exp\Big( -\frac{(10 x_1-3)^2 + 225 x_2^2}{20}\Big);
\end{align}
see \@ Figure~\ref{fig:data} for an image. The true data $\rho^0$ are then computed by a 
standard numerical method as outlined below. 
To obtain, a physically reasonable evolution, we consider instead of \eqref{eq:rho_our}--\eqref{eq:c_our} 
the system
\begin{align*}
 \partial_t \rho -D_\rho \Delta \rho&= -\div (f(\rho)\nabla c)\qquad\text{in } \Omega\times (0,T),\\
  -D_c\Delta c + A_c c &=  g(\rho) \quad\text{in } \Omega \times (0,T),
\end{align*}
with constant diffusion and absorption parameters $D_\rho=0.05$, $D_c=0.1$, and $A_c=0.01$. 
Our analytical results are valid also for this system and, with a slight abuse of notation, we will just refer to \eqref{eq:rho_our}--\eqref{eq:c_our} and \eqref{eq:rho_IP}--\eqref{eq:c_IP} below.

\subsection*{Finite Element Discretization}\label{sec:num}
In order to compute approximate solutions for \eqref{eq:rho_our}--\eqref{eq:bc}, we use a Galerkin framework.
For the discretization of $f$ we take one-dimensional continuous piecewise linear finite elements with $1000$ degrees of freedom. For the spatial discretization of $\rho$ and $c$, we employ two-dimensional continuous piecewise linear finite elements on a triangulation of $\Omega$ with $4225$ vertices, and we use a linear implicit Euler scheme with step size $\Delta t=0.025$ and $T_{end}=5$ for the time integration. Let us refer to \cite{Braess,Thomee84} for details on finite element discretizations for elliptic and parabolic problems. 

\subsection*{Simulation of the data}
Some snapshots $\rho_h^0(t)$ and $c_h^0(t)$ of the bacteria density and the concentration of the chemoattractant obtained with our simulation are depicted in Figure~\ref{fig:data}.
During the whole evolution, the range of the bacteria density $\rho_h^0$ is bounded by $\rhom^0= 7.3\times 10^{-6} \approx 0$ and $\rhoM^0= 1-1.7\times 10^{-6} \approx 1$; thus we expect that $f(\rho)$ can be identified on the whole interval $\rho \in (0,1)$, and $f^0=f^\dagger$.
%


\begin{figure}[t!]
 \includegraphics[width=0.15\textwidth]{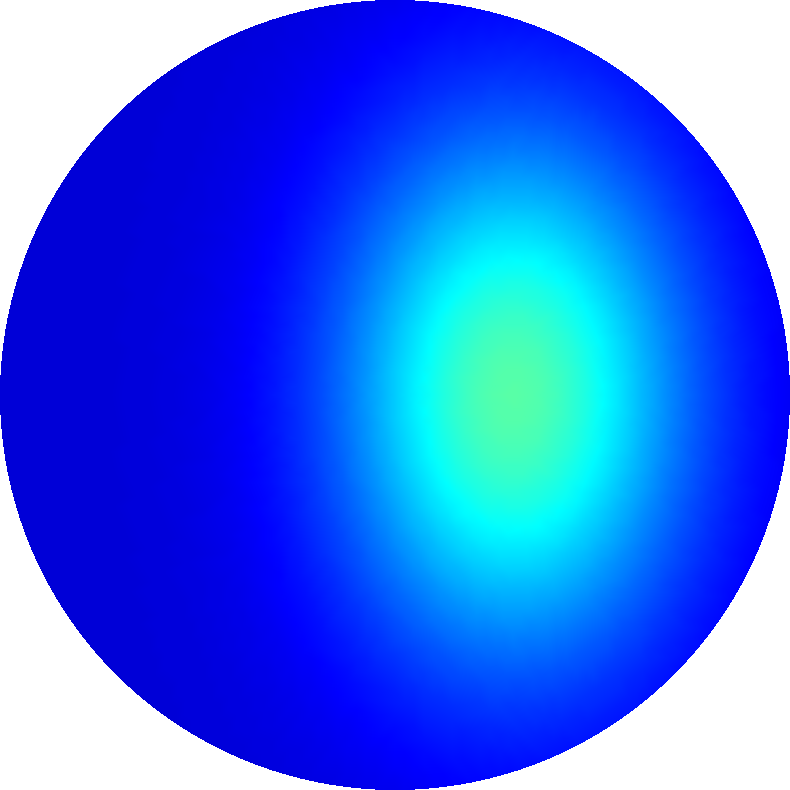}
 \includegraphics[width=0.15\textwidth]{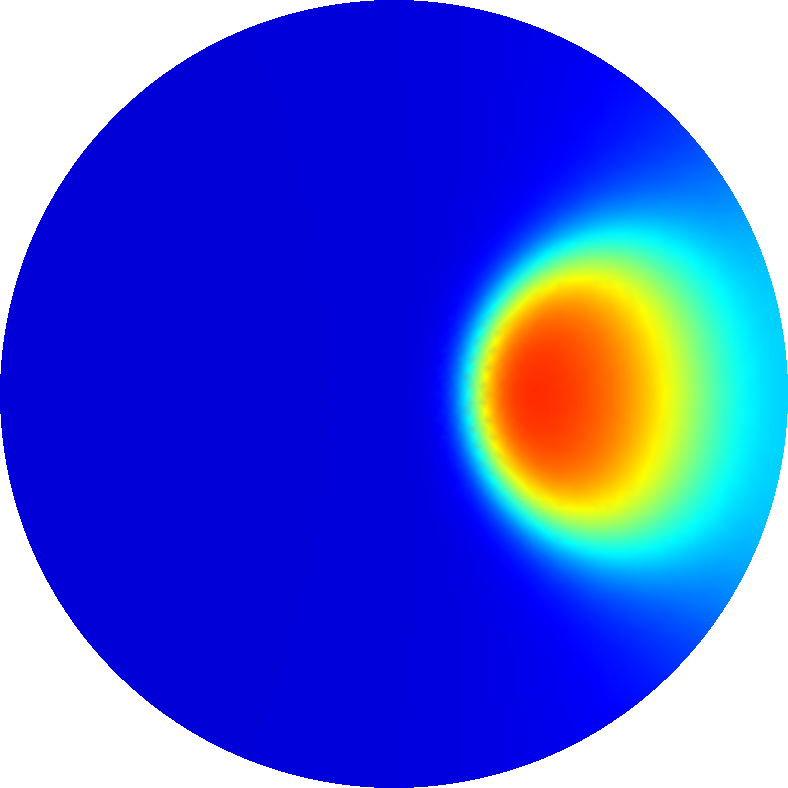}
 \includegraphics[width=0.15\textwidth]{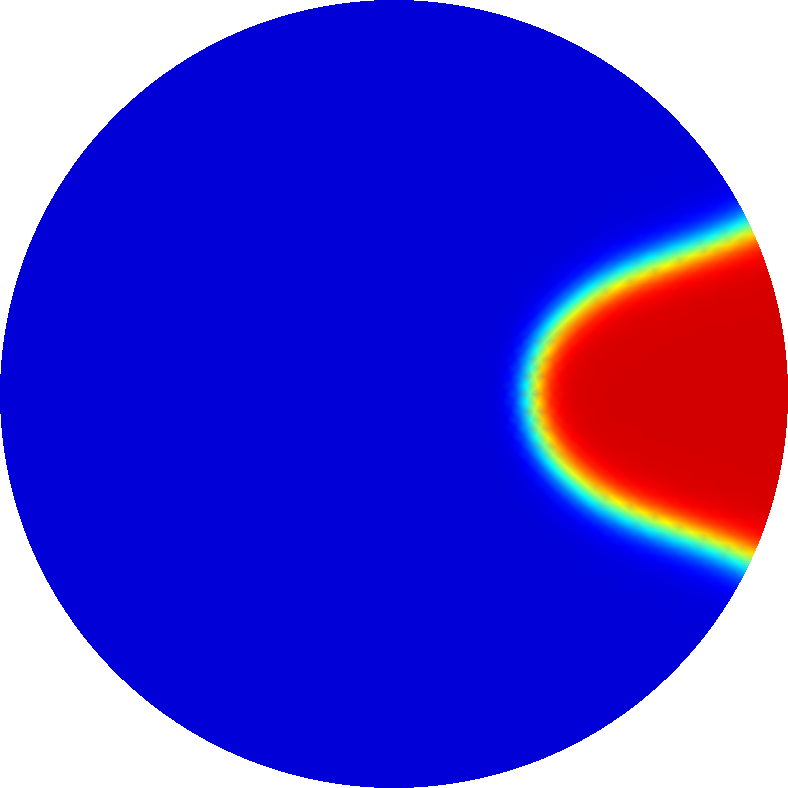}
 \includegraphics[width=0.15\textwidth]{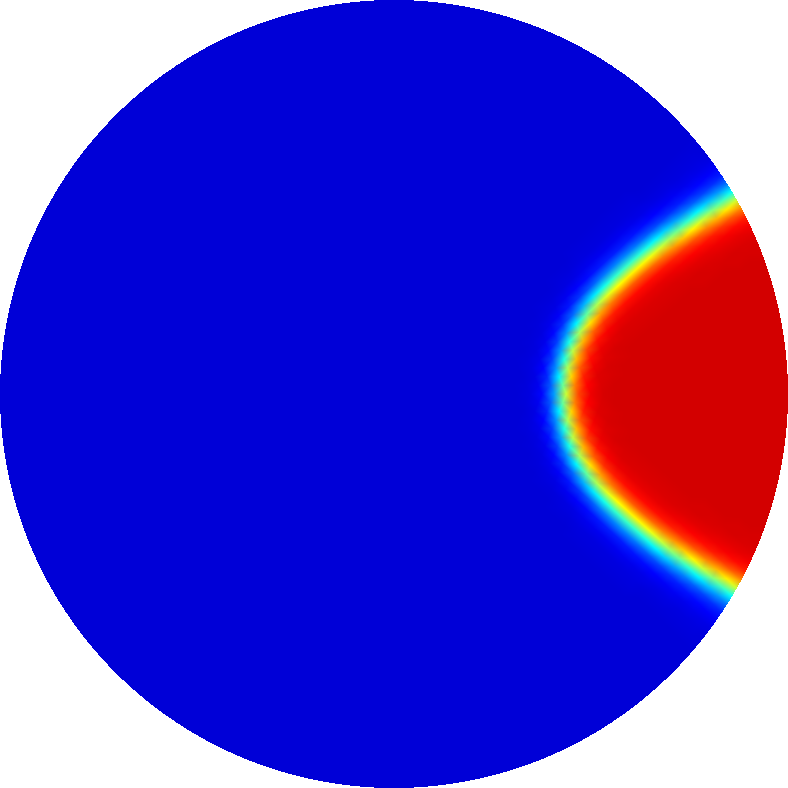}
 \includegraphics[width=0.15\textwidth]{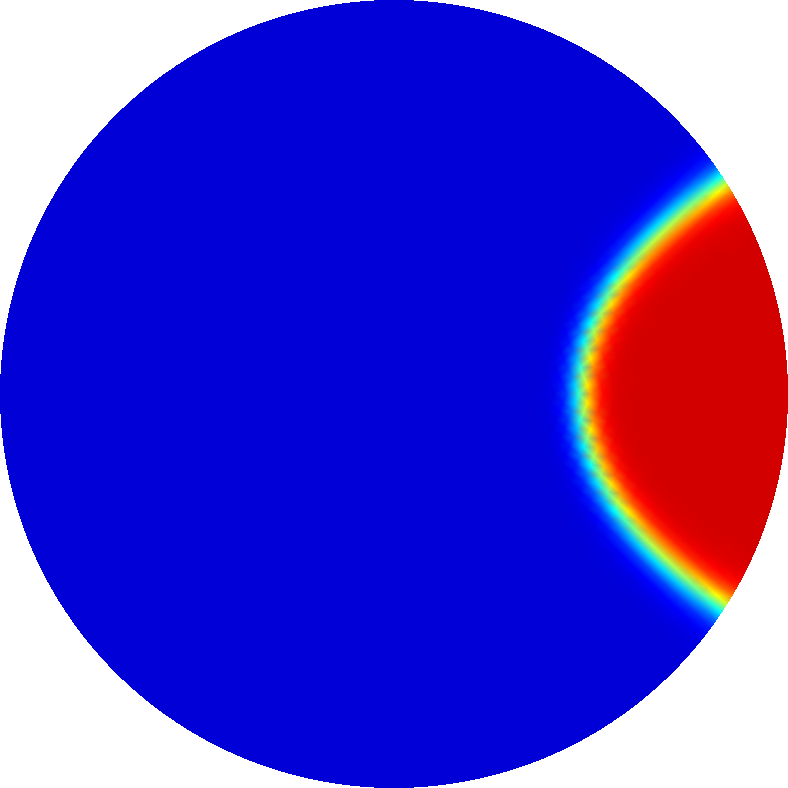}
 \includegraphics[width=0.15\textwidth]{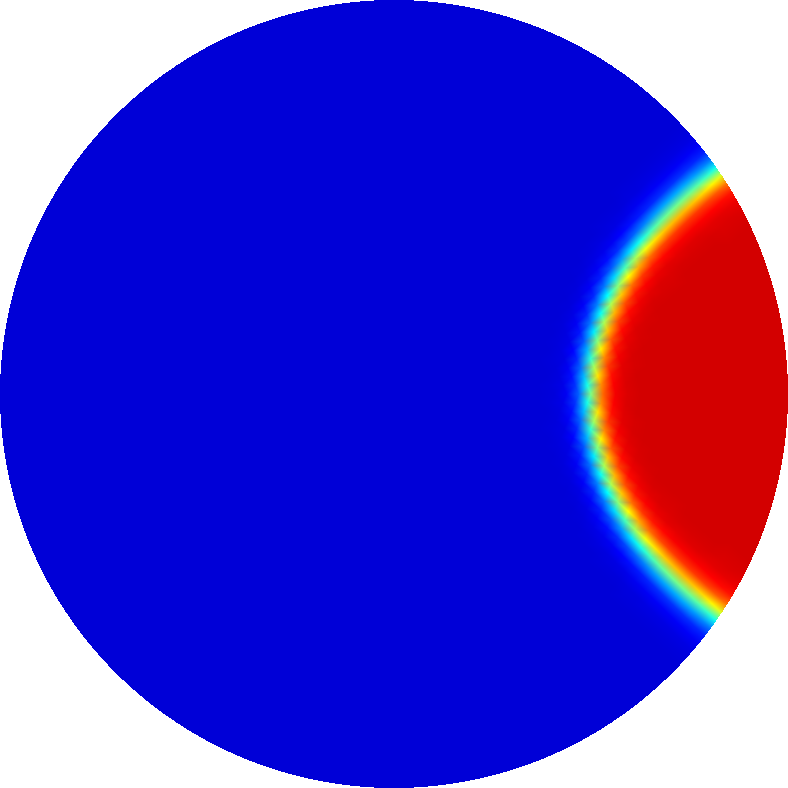}\\[0.2em]
 \includegraphics[width=0.15\textwidth]{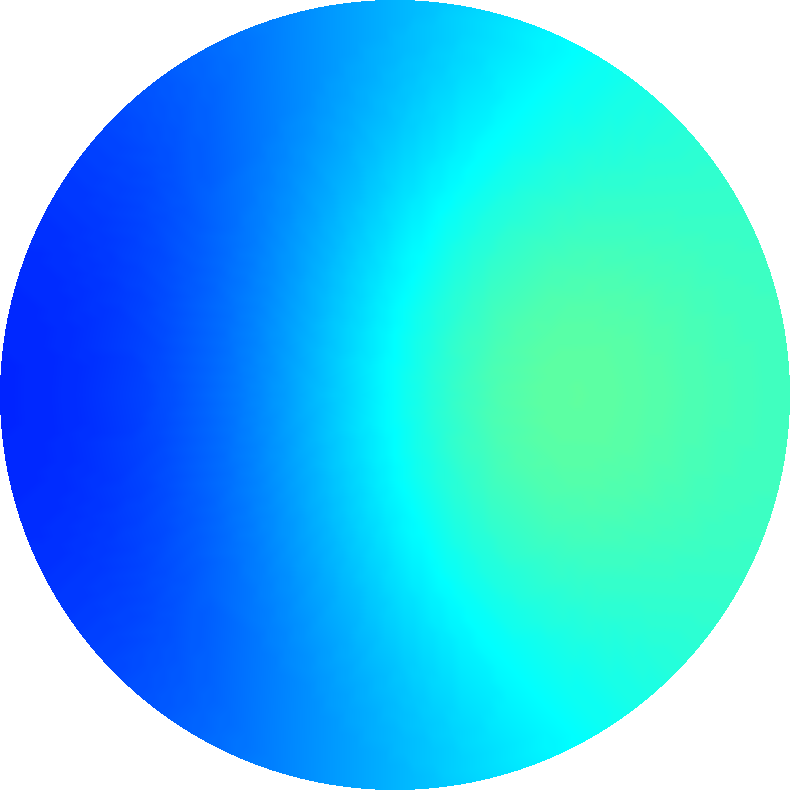}
 \includegraphics[width=0.15\textwidth]{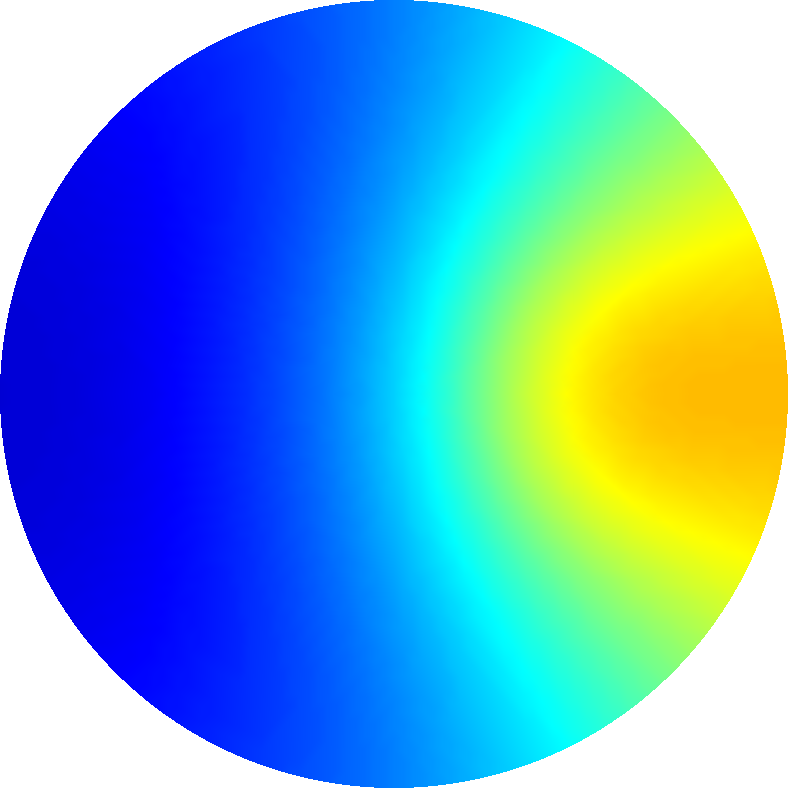}
 \includegraphics[width=0.15\textwidth]{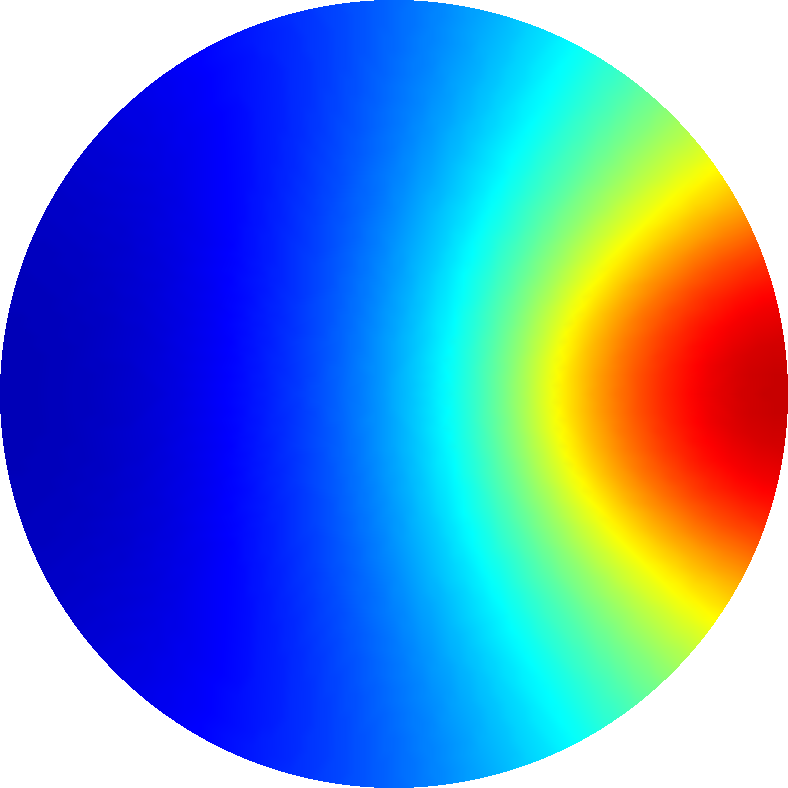}
 \includegraphics[width=0.15\textwidth]{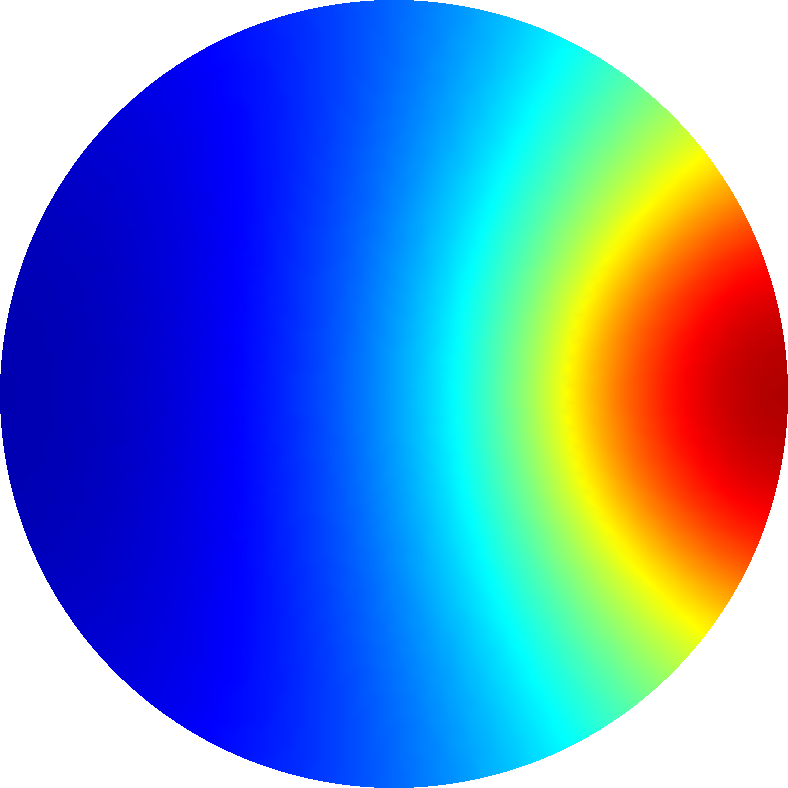}
 \includegraphics[width=0.15\textwidth]{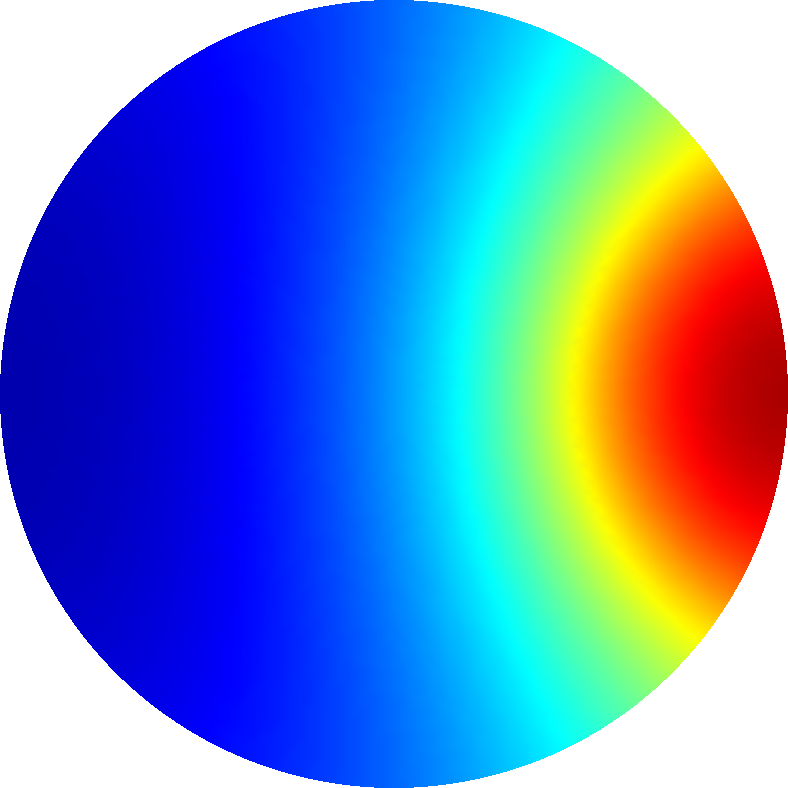}
 \includegraphics[width=0.15\textwidth]{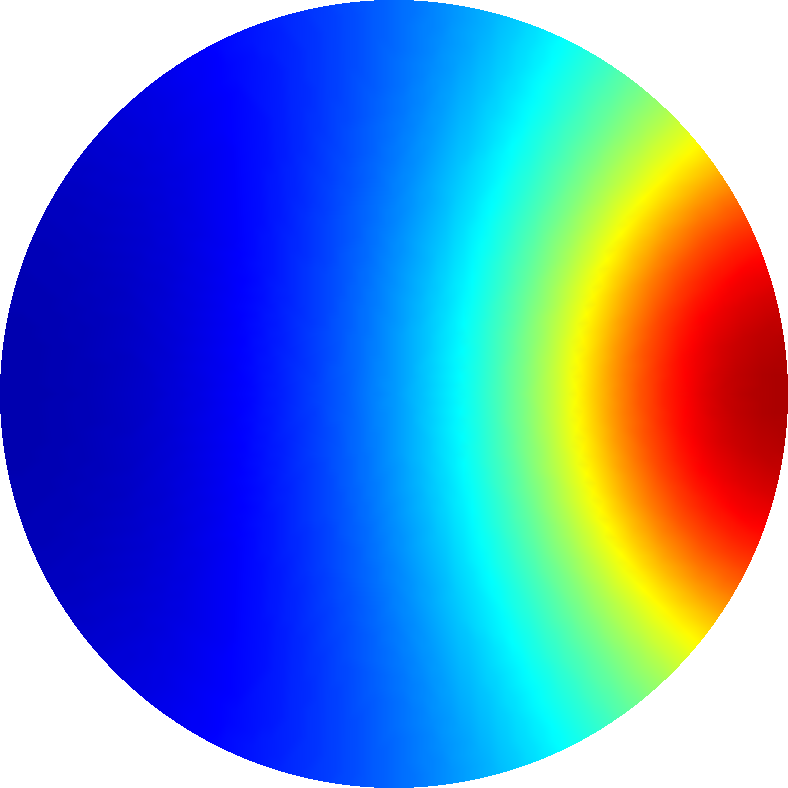}
\caption{Simulated data $\rho_h^0(t)$ (top) and corresponding concentrations $c_h^0(t)$ (bottom) for $t=0,1,\ldots,5$ from left to right.}
\label{fig:data}
\end{figure}

\subsection*{Setup of the inverse problem}

The computed data $\rho_h^0$ are perturbed by random noise such that
\begin{align*}
  \|\rho_h^0 - \rho_h^\delta\|_{L^2(0,T;L^2(\Omega))} = \delta.
\end{align*}
To obtain a discretization of the perturbed forward operator $T^\delta$, 
we proceed as follows: 
In each time step we compute $c_h(t^{n+1})$ by solving numerically the elliptic equation \eqref{eq:c_IP} with right-hand side $g(\rho_h^\delta(t^n))$. We then compute $r_h^\delta(t^{n+1})$ by solving the parabolic equation \eqref{eq:rho_IP} with right-hand side $-\nabla\cdot( f( \rho_h^\delta(t^n))\nabla c_h(t^{n+1}))$.
The discretization of the operator $T^\delta$ is then defined by the mapping $f_h \mapsto r_h^\delta$. 
The regularized approximation $f_{h,\alpha}^\delta$ is finally computed by minimizing the discrete counterpart of the Tikhonov functional $J_\alpha^\delta$ using the discrepancy principle with $\tau =1.03$ as a stopping rule. 

\subsection*{Reconstructions}
In Figure~\ref{fig:noise}, we depict the reconstructions $f_{h,\alpha}^\delta$ that were obtained for $\delta\in\{0.05,0.5\}$.
Note that we obtain rather good reconstructions already for very large noise levels, which can be explained by the fact that the inverse problem is highly overdetermined. The good quality of the reconstructions indicates that the proposed method could actually be useful in practice.

\begin{figure}[ht!]
 \includegraphics[width=0.45\textwidth]{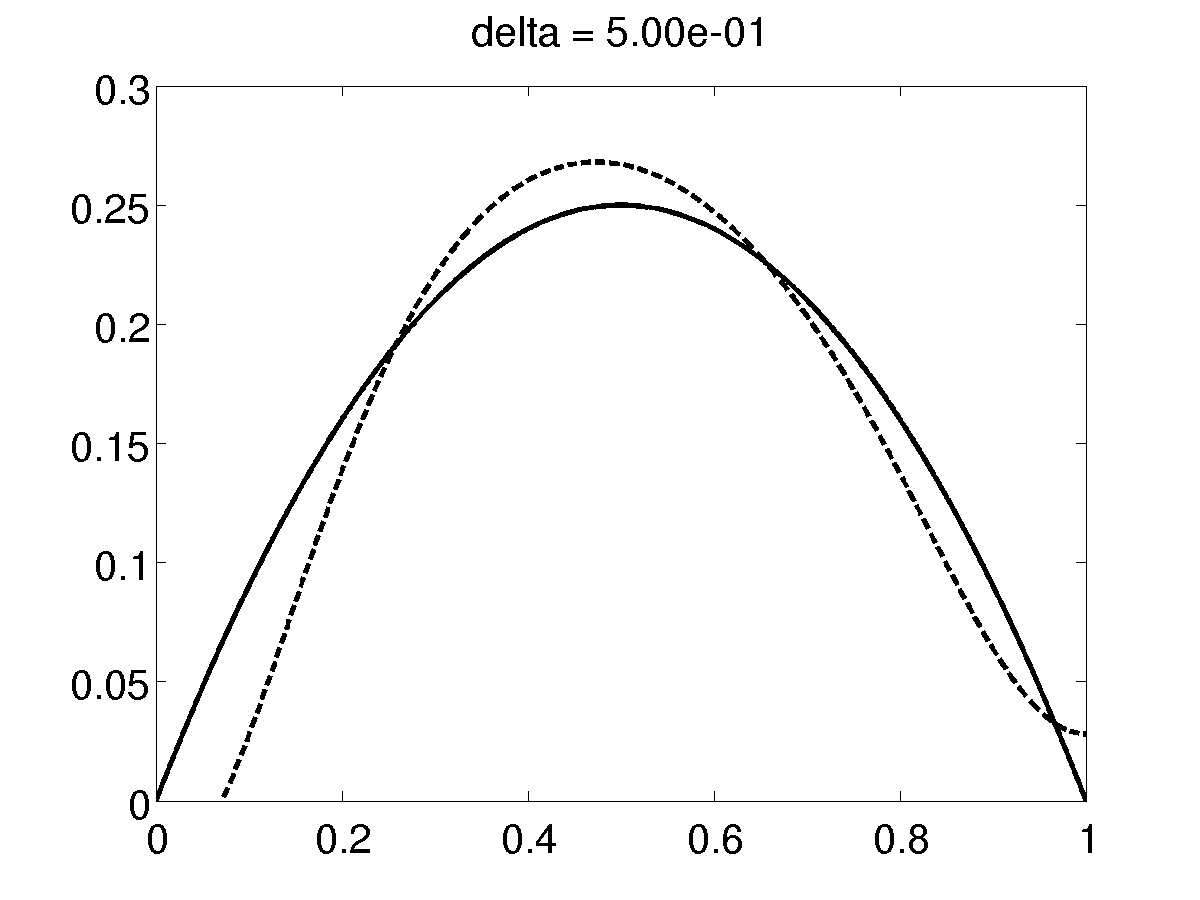}
 \hspace*{1em}
 \includegraphics[width=0.45\textwidth]{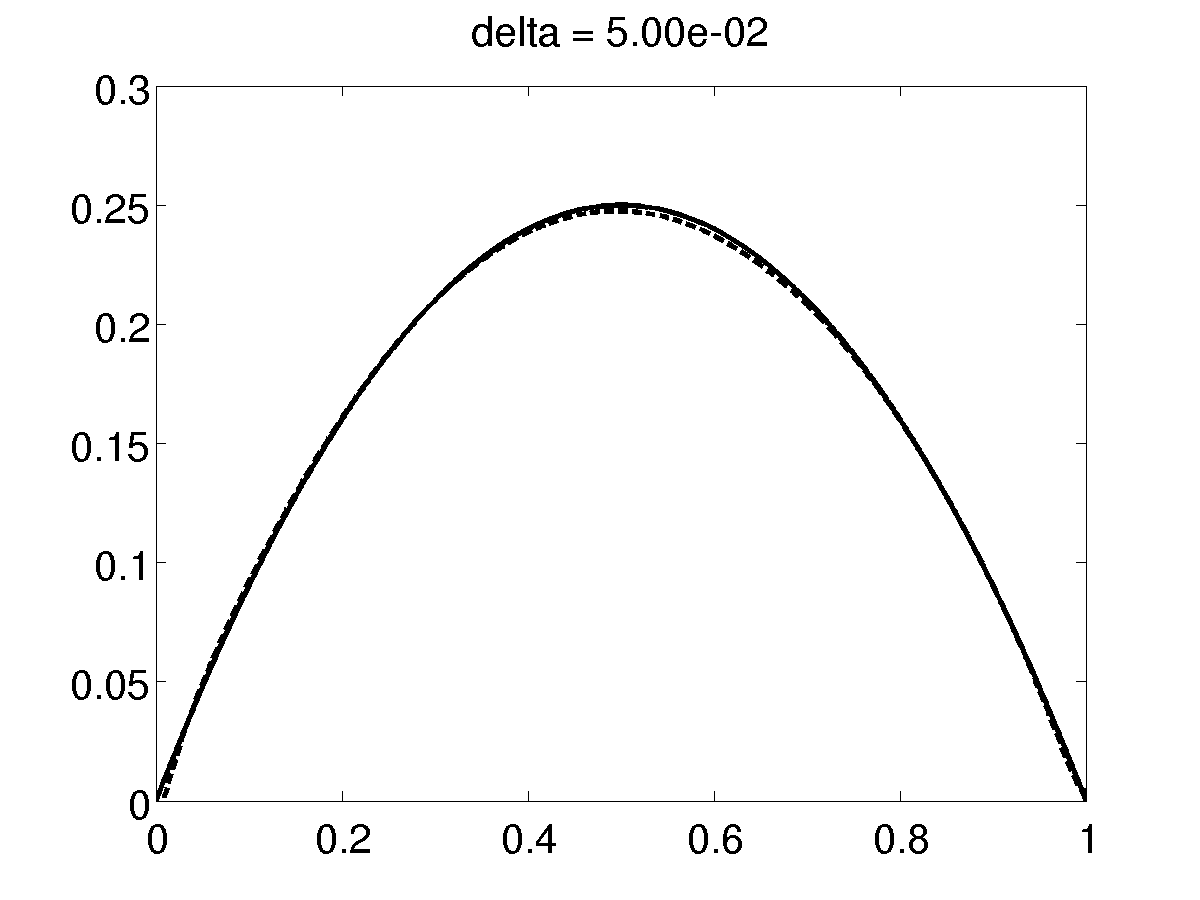}
\caption{Sensitivity $f^0(\rho)=\rho(1-\rho)$ (solid) and reconstruction $f_{h,\alpha}^\delta$ (dotted) for noise levels $\delta=0.5$ (left) and $\delta=0.05$ (right).\label{fig:noise}}
\end{figure}

In Figure~\ref{fig:rates}, we display the regularization parameters $\alpha$ chosen by the discrepancy principle, and the reconstruction errors $\|f^0-f_{h,\alpha}^\delta\|_{H^1(0,1)}$ obtained in our tests. 
As predicted by theory, when assuming that a source condition is valid, 
we observe $\alpha \approx \delta$ and $\|f^0-f_{h,\alpha}^\delta\|_{H^1(0,1)} \approx \sqrt{\delta}$ which is the best one can expect for Tikhonov regularization stopped by the discrepancy principle \cite[Chapter 5]{EHN96}.

%
\begin{figure}[ht!]
 \includegraphics[width=0.38\textwidth]{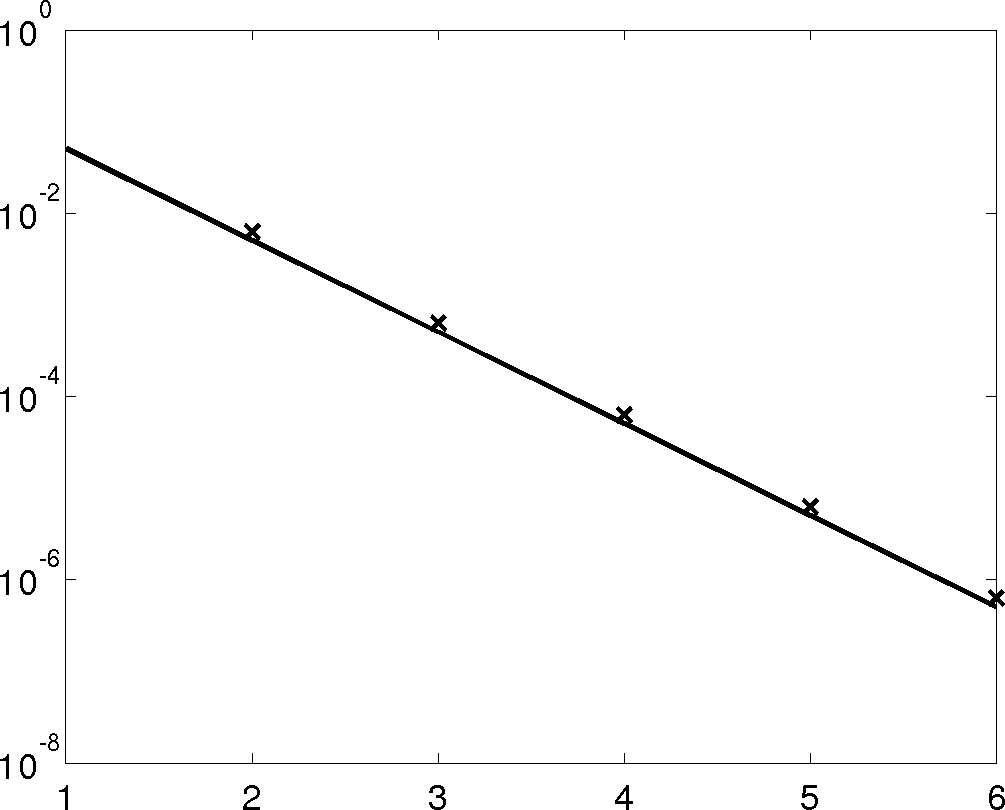}
 \hspace*{3em}
 \includegraphics[width=0.38\textwidth]{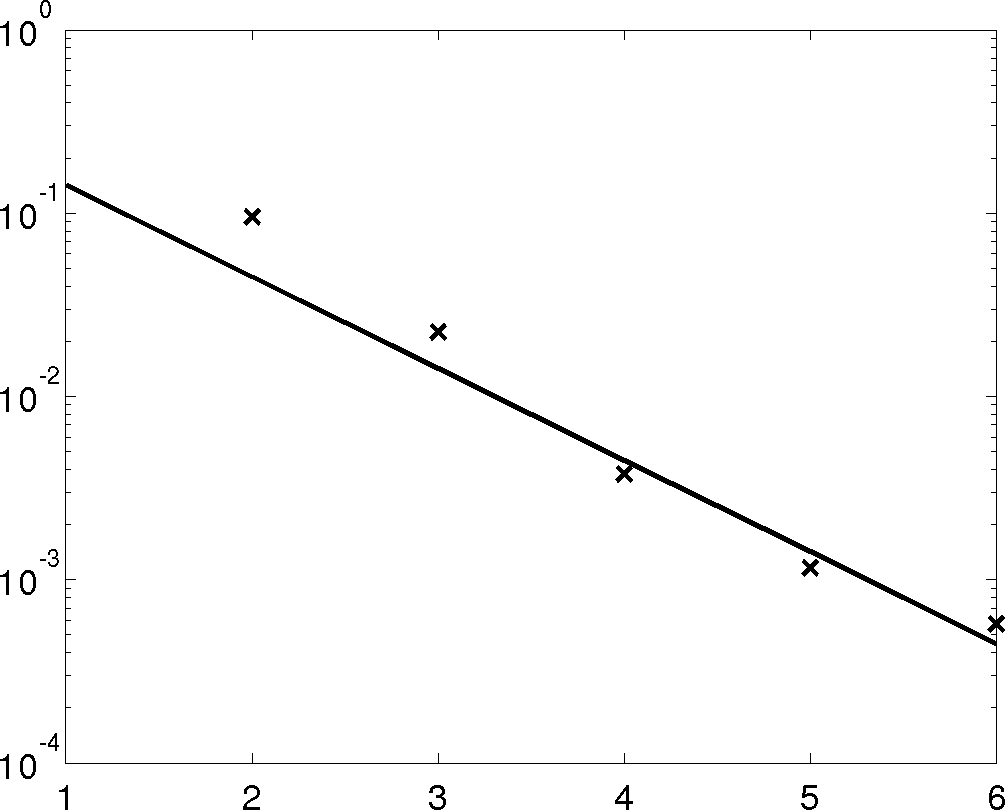}
\caption{
Regularization parameters $\alpha$ picked by the discrepancy principle for $\delta = 5\times 10^{-i}$, $i=1,\ldots,6$ (left) and reconstruction errors $\|f^0-f_{h,\alpha}^\delta\|_{H^1(0,1)}$ (right). 
The numerical results (dotted) are compared with the theoretical rates $\delta$ and $\sqrt{\delta}$, respectively. \label{fig:rates}}
\end{figure}

\section{Conclusion \& Open Problems}
In this work, we investigated the identification of the parameter functions $f(\rho)$ and $g(\rho)$ in a non-linear chemotaxis model with volume-filling. We presented uniqueness results for the identification of either parameter when the other is known from distributed measurements of the bacteria density alone. We also proposed a numerical method for actually computing the unknown functions, and illustrated its performance by numerical tests.

Let us mention some further topics of possible research concerning inverse problems in chemotaxis that could not be addressed here: 
From the theoretical point of view, the simultaneous identification of both functions $f(\rho)$ and $g(\rho)$ remains an open problem. 
A related question is, how much data is really needed to identify $f(\rho)$ and $g(\rho)$. It seems natural to conjecture that it is possible to reconstruct both functions on the range of values attained in the data, no matter how much data is available. 
Besides uniqueness, also the questions of stability of the reconstruction should be addressed. 
%
%
Our numerical results suggest that it might be possible to obtain convergence rates. It remains to verify that the chemotactic sensitivity $f$ in fact satisfies the required source condition and to interpret this condition.
Apart from the volume-filling model considered in this work, other chemotaxis models have been proposed, which also have a non-linear diffusion term, e.g. of porous medium type; see \cite{Calvez2006,Efendiev2011,Laurencot2005}. Starting from the existence theory, which is different from what we presented here, it would be interesting to see which of our results of Section~\ref{sec:inv} can be lifted to this case.
Finally, it would be interesting to see how far our results can be used to learn about real biological systems like E. coli bacteria \cite{Rajitha2009}.

\section{Acknowledgements}
HE acknowledges support by DFG via Grant IRTG 1529 and GSC 233. The work of JFP was supported by DFG via Grant 1073/1-1, by the Daimler and Benz Stiftung via Post-Doc Stipend 32-09/12 and by the German Academic Exchange Service via PPP grant no. 56052884.

\section*{Appendix}
\renewcommand{\thetheorem}{A.\arabic{theorem}}
\renewcommand{\theequation}{A.\arabic{equation}}
\setcounter{theorem}{0}
\setcounter{equation}{0}
This section summarizes some results from the linear theory of parabolic and elliptic boundary value problems that are needed in the fixed-point argument of Theorem \ref{thm:ex_ell} and elsewhere in the manuscript.

\begin{lemma}\label{lem:ex_lin_par_l2}
  For $h\in L^2(0,T;L^2(\Omega))^2$ and $u_0\in L^{2}(\Omega)$ the Neumann problem
  \begin{align}\label{eq:lin_par_l2}
    \partial_t u -\Delta u &= -\div(h)\quad\text{in }  \Omega \times (0,T),\\
    \partial_n u &= 0\quad\text{in } \partial\Omega\times(0,T),\label{eq:lin_par_bc_l2}\\
    u(0)&=u_0\quad\text{in } \Omega,
  \end{align}
  has a unique weak solution $u\in L^2(0,T; W^{1,2}(\Omega))\cap W^{1,2}(0,T;W^{1,2}(\Omega)')$, which satisfies
    \begin{align}\label{eq:lin_par_apriori_l2}
    \| u\|_{L^\infty(0,T;L^2(\Omega))}\leq \|h\|_{L^2(0,T;L^2(\Omega))} +  \|u_0\|_{L^2(\Omega)}.
  \end{align} 
  Here, the divergence has to be understood in a distributional sense, i.e.
  \begin{align*}
    \langle -{\rm div}(h), \phi \rangle := \int_0^T \int_\Omega h(x,t)\cdot \nabla\phi(x,t)\dd x \dd t \quad \text{for } \phi \in L^2(0,T;W^{1,2}(\Omega)).
  \end{align*}
\end{lemma}
\begin{proof}
The existence and uniqueness follows with standard arguments; see e.g. \cite{Evans98}. Multiplying \eqref{eq:lin_par_l2} with the solution $u$ and integrating over $\Omega \times (0,t)$ gives 
\begin{align*}
\tfrac{1}{2}\|u(t)\|_{L^2(\Omega)}^2 + \int_0^t \|\nabla u(s)\|_{L^2(\Omega)}^2 \dd s &= \tfrac{1}{2}\|u_0\|_{L^2(\Omega)}^2 + \int_0^t (h,\nabla u(s))_\Omega \dd s.
\end{align*}
The assertion then follows by an application of Young's inequality.
\end{proof}

\begin{lemma}\label{lem:ex_lin_par_lp}
 For $h\in L^q(0,T;L^q(\Omega))$, $b\in L^\infty(0,T;L^\infty(\Omega))^2$ and $u_0\in W^{2-2/q,q}(\Omega)$ with $2<q<3$, the Neumann problem
  \begin{align}\label{eq:lin_par_lp}
    \partial_t u -\Delta u + b\cdot \nabla \rho&= h\quad\text{in } \Omega \times (0,T),\\\label{eq:lin_par_bc_lp}
    \partial_n u &= 0\quad\text{in } \partial\Omega\times(0,T),\\
  u(0)&=u_0\quad\text{in } \Omega,
  \end{align}
  has a unique solution $u\in L^q(0,T;W^{2,q}(\Omega))\cap W^{1,q}(0,T;L^q(\Omega))$ which satisfies
  \begin{align*} 
    \|u\|_{L^q(0,T;W^{2,q}(\Omega))} + \|\partial_t u\|_{L^q(0,T;L^q(\Omega))} 
    &\le C(\|h\|_{L^q(0,T;L^q(\Omega))} + \|u_0\|_{W^{2-2/q,q}(\Omega)}).
  \end{align*}
  In particular, we deduce from Sobolev embeddings that $u \in C^0(\overline{\Omega}\times [0,T])$.
  \end{lemma}
For the proof, let us refer to \cite{Ladyshenskaja68}.
\begin{lemma}\label{lem:ex_lin_ell}
Let $h\in L^\infty(0,T;L^q(\Omega))$ for some $2\leq q <\infty$. Then there exists a unique $u\in L^\infty(0,T;W^{2,q}(\Omega))$ satisfying the Neumann problem
  \begin{align}\label{eq:lin_ell}
  -\Delta u + u&= h\quad\text{in } \Omega,\,\text{ for a.e. }t\in[0,T]\\\label{eq:lin_ell_bc}
    \partial_n u &= 0\quad\text{on } \partial\Omega,\,\text{ for a.e. }t\in[0,T].
  \end{align}
Moreover, the following a-priori estimates hold
\begin{align}
 \|u\|_{L^\infty(0,T;W^{2,q}(\Omega))} &\le C\|h\|_{L^\infty(0,T;L^q(\Omega))}, \label{eq:apriori_c}\\
 \|u\|_{L^r(0,T;W^{1,2}(\Omega))} &\le \|h\|_{L^r(0,T;L^2(\Omega))},  \label{eq:apriori_c2}
\end{align}
where $C$ only depends on $\Omega$ and $q$, and $1\leq r\leq \infty$ is arbitrary.
\end{lemma}
\begin{proof}
Existence of a unique solution in $W^{2,q}(\Omega)$ for a.e. $t\in [0,T]$ follows from standard arguments in the theory of linear elliptic equations, see e.g. \cite[Thm. 2.4.2.7]{Grisvard85}. The a-priori estimate \eqref{eq:apriori_c} follows by the bounded inverse theorem and by taking the supremum over $t$. Estimate \eqref{eq:apriori_c2} follows in a similar fashion as the a-priori estimate in Lemma~\ref{lem:ex_lin_par_l2}.
\end{proof}

\bibliographystyle{abbrv}
\bibliography{bib}

\end{document}